\documentclass[a4paper,12pt]{article}
\usepackage{amssymb}
\usepackage{amsthm}
\usepackage{amsmath}
\usepackage{amsfonts}
\usepackage{soul}
\usepackage{xcolor}
\usepackage{enumitem}
\setlist[itemize]{noitemsep, nolistsep}
\setlist[enumerate]{noitemsep}

\begin{document}

\newcommand\sect{\section}

\newtheorem{thm}{Theorem}[section]
\newtheorem{cor}[thm]{Corollary}
\newtheorem{lem}[thm]{Lemma}
\newtheorem{prop}[thm]{Proposition}
\newtheorem{propconstr}[thm]{Proposition-Construction}
\newtheorem{fact}[thm]{Fact}

\theoremstyle{definition}
\newtheorem{para}[thm]{}
\newtheorem{ax}[thm]{Axiom}
\newtheorem{conj}[thm]{Conjecture}
\newtheorem{defn}[thm]{Definition}
\newtheorem{notation}[thm]{Notation}
\newtheorem{rem}[thm]{Remark}
\newtheorem{remark}[thm]{Remark}
\newtheorem{question}[thm]{Question}
\newtheorem{example}[thm]{Example}
\newtheorem{problem}[thm]{Problem}
\newtheorem{exercise}[thm]{Exercise}
\newtheorem{ex}[thm]{Exercise}

\newcommand{\red}[1]{\textcolor{red}{#1}}
\newcommand{\blue}[1]{\textcolor{blue}{#1}}
\newcommand{\green}[1]{\textcolor{green}{#1}}
\newcommand{\orange}[1]{\textcolor{orange}{#1}}
\newcommand{\si}{\sigma}
\newcommand{\prf}{\smallskip\noindent{\it        Proof}. }

\newcommand{\la}{\langle}
\newcommand{\ra}{\rangle}
\newcommand{\inv}{^{-1}}
\newcommand{\trdeg}{{\rm tr.deg}}

\newcommand{\rest}{{\lower       .25     em      \hbox{$\vert$}}}
\newcommand{\ch}{{\rm    char}}
\newcommand{\zee}{{\mathbb  Z}}
\newcommand{\conc}{^\frown}
\newcommand{\acl}{{\rm acl}}
\newcommand{\dcl}{{\rm dcl}}
\newcommand{\cl}{{\rm cl}}

\newcommand{\mult}{{\rm  Mult}}

\newcommand{\aut}{{\rm   Aut}}
\newcommand{\Aut}{{\rm Aut}}

\newcommand{\cala}{{\mathcal A}}
\newcommand{\calb}{{\mathcal B}}
\newcommand{\calc}{{\mathcal C}}
\newcommand{\cald}{{\mathcal D}}
\newcommand{\cale}{{\mathcal E}}
\newcommand{\calf}{{\mathcal F}}
\newcommand{\calg}{{\mathcal G}}
\newcommand{\calh}{{\mathcal H}}
\newcommand{\cali}{{\mathcal I}}
\newcommand{\calj}{{\mathcal J}}
\newcommand{\calk}{{\mathcal K}}
\newcommand{\call}{{\mathcal L}}
\newcommand{\calm}{{\mathcal M}}
\newcommand{\caln}{{\mathcal N}}
\newcommand{\calo}{{\mathcal O}}
\newcommand{\calp}{{\mathcal P}}
\newcommand{\calq}{{\mathcal Q}}
\newcommand{\calr}{{\mathcal R}}
\newcommand{\cals}{{\mathcal	S}}
\newcommand{\calt}{{\mathcal T}}
\newcommand{\calu}{{\mathcal U}}
\newcommand{\Uu}{{\mathcal U}}
\newcommand{\calv}{{\mathcal V}}
\newcommand{\calw}{{\mathcal W}}
\newcommand{\calx}{{\mathcal X}}
\newcommand{\caly}{{\mathcal Y}}
\newcommand{\calz}{{\mathcal Z}}

\newcommand{\aff}{{\mathbb A}}

\newcommand{\rat}{{\mathbb Q}}
\newcommand{\ga}{{\mathbb G}_a}
\newcommand{\gm}{{\mathbb G}_m}
\newcommand{\cee}{{\mathbb C}}
\newcommand{\nat}{{\mathbb  N}}
\newcommand{\ree}{{\mathbb R}}

\newcommand{\frob}{{\rm Frob}}
\newcommand{\Frob}{{\rm Frob}}
\newcommand{\fix}{{\rm Fix}}
\newcommand{\Fix}{{\rm Fix}}
\newcommand{\proj}{{\mathbb P}}
\newcommand{\sym}{{\rm Sym}}
\newcommand{\gal}{{\mathrm{Gal}}}
\newcommand{\Gal}{{\mathrm Gal}}
\newcommand{\loc}{{\rm locus}}
\newcommand{\Kk}{{\mathbb K}}

\newcommand{\Ker}{{\rm Ker}\,}
\newcommand{\cb}{{\rm Cb}}
\newcommand{\bigudot}{\hbox{$\bigcup\kern-.75em\cdot\;\,$}}

\newcommand{\sime}{{\hbox{$\kern-.19em\sim$}}}

\newcommand{\acb}{\overline{\rm Cb}}
\newcommand{\dnfo}{\,\raise.2em\hbox{$\,\mathrel|\kern-.9em\lower.35em\hbox{$\smile$}
$}}
\newcommand{\dfo}{\;\raise.2em\hbox{$\mathrel|\kern-.9em\lower.35em\hbox{$\smile$}
\kern-.7em\hbox{\char'57}$}\;}

\newcommand{\Hom}{{\rm Hom}}


\newcommand\vlabel{\label}

\title{Measures on  perfect $e$-free PAC fields}
\author{Zo\'e Chatzidakis\thanks{partially supported by the ANR AGRUME
    (ANR-17-CE40-0026) and by the ANR 
GeoMod  (ANR-DFG, AAPG2019)}{ (CNRS  - ENS)} 
  \and
  {Nicholas Ramsey}{(UCLA)}}
\maketitle
 \begin{abstract}
We construct measures on definable sets in $e$-free perfect PAC fields,
as well as on perfect PAC fields whose absolute Galois groups
are free pro-$p$ of finite rank.  We deduce the definable amenability of all groups definable in such fields.  As a corollary, we additionally prove the definable amenability of all groups definable in perfect $\omega$-free PAC fields via ultralimit measures.  
 \end{abstract}

\section{Introduction}

In this paper, we construct measures on definable sets in perfect
pseudo-algebraically closed fields with free absolute Galois group.  A
field $K$ is called \emph{pseudo-algebraically closed}, or \emph{PAC},
if every absolutely irreducible variety defined over $K$ has a
$K$-rational point.  The class of PAC fields was isolated by Ax (\cite{A}) in
the course of algebraically characterizing pseudo-finite fields; he
showed that the class of pseudo-finite fields coincides with the class
of perfect PAC fields with free absolute Galois group of rank 1.
In the early 80's, PAC
fields were given systematic study by Cherlin, van den Dries, and
Macintyre (\cite{CDM81}) and separately by Ershov (\cite{E}).  Since then, they have been a
central object of study in both model theory and field arithmetic.   

Our work was motivated by two results and one
question. In \cite{CDM}, van den Dries, Macintyre, and the first 
author showed how to definably associate a measure to each
definable set in a pseudo-finite field. This measure is a {\em
  non-standard counting measure}. Hrushovski shows in \cite{H} that this
counting measure is the
only finitely additive probability measure  on a pseudo-finite field which satisfies Fubini. This result was
later extended by Halupczok (\cite{IH}) who replaced the Fubini
condition by preservation under definable bijections and some
invariance condition,  and showed, in essence, that in order for a perfect PAC field to have
a probability measure on definable sets satisfying these conditions, the absolute
Galois group of the field had to be procyclic. So, in a PAC field with
non-procyclic absolute Galois group, it is natural to ask if there is still a way to construct a useful measure on
definable sets, which necessarily has to satisfy weaker
conditions. It is always possible to construct \emph{some}
  measures on definable sets from types, which determine
  $\{0,1\}$-valued measures, but these are usually less relevant for the
  kinds of applications for which measures have been used. A test
  question for our measures  {originally came} from the
  problem of determining if all groups in simple theories are definably
  amenable. 
  Recall that a
definable group $G$ is \emph{definably amenable} if there is a
left-translation invariant finitely additive probability measure on
definable subsets of $G$.  It is known that all stable groups are
definably amenable.   Within the class of simple theories it was only
recently discovered that there are non-definably amenable groups with
simple theories (\cite{CHKKMPR}).  Since groups definable in bounded
perfect PAC fields form one of the most interesting classes 
of groups with simple theories, it is natural to ask if they are all
definably amenable.  

In the particular case of a perfect $e$-free PAC field $k$ (i.e., with
absolute Galois group isomorphic to $\hat F_e$, the free profinite group
on $e$ generators), we are able to define a measure on definable
subsets of (absolutely irreducible) varieties, which coincides with the
non-standard counting measure when $e=1$. Our work uses in an essential
way the measure introduced by Jarden and Kiehne (\cite{JK}). We build on their results by using their measure on sentences to define measures on definable sets.

We also give a description of how
the measure changes under definable maps, and  show
that it is preserved under birational automorphisms of a variety.  This
latter result applies to show that if $G$ is an algebraic group defined
over $k$, then
$G(k)$ is definably amenable. Using results of Hrushovski and Pillay
\cite{HP1} showing that all groups definable in an $e$-free PAC field are
virtually isogenous to algebraic groups, we are then able to deduce the
definable amenability of all groups definable in perfect $e$-free PAC
fields. 

Our measures are constructed explicitly for definable subsets of perfect
PAC fields with free absolute Galois group of \emph{finite} rank, but this has
useful consequences for perfect $\omega$-free PAC fields as well.  Recall that a
PAC field $K$ is called $\omega$-free if it is elementarily equivalent
to a PAC field with absolute Galois group $\hat{F_{\omega}}$.  Every $\omega$-free
PAC field is elementarily equivalent to a non-principal ultraproduct of $e$-free PAC
fields ($e\in\nat$).  Thus,
from the definable amenability of all groups definable in $e$-free PAC
fields, we obtain that groups definable in $\omega$-free PAC fields are
  definably amenable,  via the ultralimit measure.  On the algebraic
  group itself, the measure is a $0-1$ measure, so not very interesting
  in itself. It raises however interesting questions on the behaviour of
  the measure on
  non-algebraic definable groups.
As $\omega$-free PAC
fields are a core example of NSOP$_{1}$ theories, for which a
satisfactory theory of groups is not yet available, we hope that these
results will play a useful role in classifying groups definable in $\omega$-free PAC
fields.  

We are also able to generalize our results to the case of perfect
  PAC field with free pro-$p$ absolute Galois group.  This involves a similar construction, but working over the fixed field of a fixed $p$-Sylow subgroup.  {As a corollary, we obtain definable amenability for groups definable in these fields as well. 
  
This paper is organized as follows.  Section 2 establishes several
preliminary facts about definable sets and types in PAC fields, as well
as a description of the aforementioned measure of Jarden and Kiehne.
The main construction of measures on definable subsets of a perfect
$e$-free PAC field takes place in Section 3, and we prove
  there that all groups definable in perfect $e$-free PAC fields are definably amenable.  In Section 4, we describe extensions to other fields, namely, the perfect $\omega$-free PAC fields and the perfect PAC fields with free pro-$p$ absolute Galois group.  We conclude this section with some questions about the behaviour of these measures and possible extensions.  
\section{Preliminaries}

\begin{defn} Let $e\geq 1$ be an integer. A field $K$ is
{\em $e$-free PAC} if it
is  PAC, and has absolute Galois group $\gal(K)$ isomorphic to
$\hat F_e$ (the profinite completion of the free group on $e$
generators).\\ 
Let $\call=\{+,-,\cdot,0,1\}$ be the language of rings, $E$ a field, and
$\call(E)$ the language $\call$ augmented by constant symbols for the
elements of $E$. A {\em test sentence\footnote{Also called a
  {\em one-variable statement} in \cite{Jw}.} over $E$} is a Boolean
combination of $\call(E)$-sentences of the form $\exists y\, f(y)=0$,
where $f(y)\in E[y]$ ($y$ a single variable). Note that in particular any quantifier-free
sentence of $\call(E)$ is a test sentence. \\
Similarly, we say that $\theta(x)$ is a {\em test formula
  formula over $E$} (in the tuple of variables $x$)  if it is a Boolean combination of
$\call(E)$-formulas of the form $\exists y\, f(x,y)=0$, where $y$ is a
single variable, and $f\in E[x,y]$. 

\end{defn}

\begin{notation} If $E$ is a field, then $E^{alg}$ denotes an algebraic closure of $E$, and $E^s$ its separable
closure. We will view the absolute Galois group $\gal(E)$ as acting on
$E^{alg}$. ${\rm Var}_k$ denotes the set of (absolutely irreducible)
varieties defined over $k$. 
\end{notation}

\begin{fact}\vlabel{fact1} (Corollary 20.4.2 and Lemma 20.6.4 in \cite{FJ}) Let $K$ and $L$ be
  perfect $e$-free PAC fields, and $E$ a common subfield. Then
  $$K\equiv_EL \iff E^s\cap K\simeq_E E^s\cap L,$$
if and only if they satisfy the same test sentences over $E$.   
  \end{fact}

An immediate application gives a description of types:
\begin{cor}\vlabel{types} Let $E$ be a subfield of an $e$-free perfect
  PAC field $\calk$, and let $a$, $b$, be tuples in $\calk$. The
  following conditions are equivalent:
  \begin{enumerate}
  \item $tp(a/E)=tp(b/E)$;
    \item There is an $E$-isomorphism $E(a)^s\cap \calk\to E(b)^s\cap
      \calk$ which sends $a$ to $b$;
    \item $a$ and $b$ satisfy the same test-formulas over $E$;
   \item for every finite Galois extension $L$ of $E(a)$, there is a
     field embedding $\varphi:L\to E(b)^s$ such that $\varphi(L\cap
     \calk)=\varphi(L)\cap \calk$.   
 \end{enumerate}
 \end{cor}
  \begin{rem}\vlabel{theta} Given  a finite Galois extension $L$ of $E(a)$, and  a
subfield $K$ of $L$ containing $E(a)$, we will build a  test-formula 
 $\theta_K(x)$ over $E$ as follows: 
 for
each subfield $F$ of $L$ containing $K$, select a generator $\alpha_F$
over $E(a)$, and let $P_F\in E[x][y]$ be such that $P_F(a,y)$ is a 
minimal polynomial of $\alpha_F$ over
$E(a)$; then the formula $\theta_K(x)$ says the following:
$$\exists y\, P_K(x,y)=0 \land \bigwedge_{F\neq K} \forall y\,
P_F(x,y)\neq 0 \land Q(x)\neq 0,$$
where $Q(x)$ is the product of all leading coefficients in $y$ of the
polynomials $P_F(x,y)$. This formula has the following property:
whenever $\calk$ contains $E(a)$, then
$$\calk\models \theta_K(a)\iff \calk\cap L\simeq K.$$
    \end{rem}


\para{\bf The measure of Jarden and Kiehne}. \vlabel{JK} Let $K$ be a countable
Hilbertian field, let $e\geq
1$, and consider the Haar measure on $\gal(K)^e$. View $\gal(K)$ as acting on
$K^{alg}$. For each $\call(K)$-sentence $\theta$, they define
$$\mu(\theta)=\mu(\{\bar\si\in
\gal(K)^e\mid \Fix(\bar\si)\models \theta\}).$$
Then, for almost all $\bar\si\in \gal(K)^e$, $\Fix(\bar\si)$ is $e$-free
PAC. Again, by the elementary invariants of $e$-free PAC fields, for
every $\call(K)$-sentence $\theta$, there is a finite Galois extension
$L$ of $K$, and finitely many subextensions $M_1,\ldots,M_r$ of $L$ such
that for all $\bar\si\in \gal(K)^e$, $$\Fix(\bar\si)\models \theta \iff
\bigvee_i\Fix(\bar\si)\cap L\simeq_{K}M_i.$$  Hence, assuming that if $i\neq j$, then $M_i\not\simeq_K M_j$,
we have 
$$\mu(\theta)=\sum_i|\{\bar\si\in
\gal(L/K)^e\mid \Fix(\bar\si)\simeq_{K}M_i\}|\,[L:K]^{-e}.$$

\begin{fact} \vlabel{Gaschutz} (An ingredient in the proof of Gasch\"utz
  Lemma --   Lemma 17.7.1 of \cite{FJ}.) Let $f:B\to A$ be an
epimorphism of finite groups, with $B$ $e$-generated. If $\bar h\in
A^e$ generates $A$, then $\{\bar g\in B^e\mid f(\bar g)=\bar h,\, \langle
\bar g \rangle=B\}$ is non-empty, and its size does not depend on $\bar
h$.  
\end{fact}

\begin{notation}  Let $K\leq E\leq F$, with $F$ and $E$ finite Galois over $K$,
and $\gal(F/K)$ $e$-generated. We denote by $\phi(E,K)$, or $\phi(E)$,
the number of $\bar\si\in \gal(E/K)^e$ which generate $\gal(E/K)$. This
only depends on $e$ and the group $\gal(E/K)$. We denote by $\phi(F/E,K)$
or $\phi(F/E)$ the number given by Gasch\"utz Lemma, i.e.: given a set
$\bar\si\in \gal(E/K)^e$ generating $Gal(E/K)$, the number of lifts
$\bar\tau\in \gal(F/K)^e$ of $\bar\si$ which generate $\gal(F/K)$. Note
that $\phi(F)=\phi(F/E)\phi(E)$.
\end{notation}

\section{Definition of the measure when $k$ is $e$-free PAC}

\para{\bf Setting and notation}. We let $e\in\nat$, $e\geq 1$,  $k$ be a countable perfect
field with $\gal(k)\simeq \hat F_e$, and $V \in {\rm Var}_{k}$. We
will define a measure $\mu_V$ on definable subsets of $V$. We fix a
generic $a$ of $V$ over $k$. If $S\subset V$ is
definable and not
Zariski dense in $V$, then we set $\mu_V(S)=0$. \\ 
If $L$ is a finite Galois extension of $k(a)$, we let $k_L=k^{alg}\cap
L$. \\ 
We set $(\si_1,\ldots,\si_e)\cdot
(\tau_1,\ldots,\tau_e)=(\si_1\tau_1,\ldots,\si_e\tau_e)$. 

\para{\bf Definition of the measure}. 
Using the description of types (\ref{types}), it follows
that for every $k$-formula $\varphi(x)$, there is a finite Galois
extension $L$ of $k(a)$ such that whether $a$ satisfies $\varphi$ or not
in some $e$-free perfect PAC regular extension $\calk$ of $k$, 
only depends on the $k(a)$-isomorphism type of $\calk\cap L$. Let $L$ be finite Galois over $k(a)$, and let 
$K_1,\ldots,K_r$ be the subfields of $L$ containing $k(a)$, which are
regular over $k$ and with $\gal(L/K_i)$ an image of $G(k)$. We take them up to
conjugation over $k(a)$. For each $i$, consider the 
 $\call(k)$-formula $\theta_i(x)=\theta_{K_i}(x)$, see \ref{theta}. It
defines a subset of $V(k)$. 

Since these events are mutually exclusive on the generics of $V$,
it suffices to assign a measure to each $\theta_i$, and  compute
$\mu_V(\varphi)$ as the sum of the appropriate
$\mu_V(\theta_i)$. Indeed, they will be mutually exclusive on a
  Zariski open subset of $V(k)$, the complement of which will have
  measure $0$. See also Proposition 20.6.6 in \cite{FJ}. Letting $\mu$ denote the measure of Jarden and Kiehne
(see \ref{JK}), we
define

$$\mu_V(\theta_i)=\frac{\mu(\theta_i(a))}{\sum_{j}\mu(\theta_j(a))}.$$
Note that
\begin{align*}\sum_{j}\mu(\theta_j(a))&=\mu(\{\bar\si\in \gal(k(a))^e\mid
  \Fix(\bar\si)\cap k^s=k\})\\
  &=|\{\bar\si\in \gal(L/k(a))^e\mid
\Fix(\bar\si)\cap k_L=k\}|[L:k(a)]^{-e},\end{align*}
i.e., those $\bar\si$ whose fixed field (within $L$) is regular over
$k$.
We do this for every finite Galois extension of $k(a)$. Another way of
expressing $\mu_V(\theta_i)$ as defined above is simply as
$$\mu_V(\theta_i)=\frac{|\{\bar\si\in \gal(L/k(a))^e\mid \Fix(\bar\si)\simeq_{k(a)}K_i\}|}{|\{\bar\si\in \gal(L/k(a))^e\mid
\Fix(\bar\si)\cap k_L=k\}|}.$$




\begin{prop} The measure $\mu_V$ is well-defined.
\end{prop}

\prf We need to show that the measure does not depend on the choice of
the Galois extension $L$. 
I.e., that if $M$ contains $L$, and we do the counting in $\gal(M/k(a))$ 
to compute $\mu_V(\theta_i)$, we obtain the same number. This will
follow from the following Claim, which we will prove
below:\\[0.05in]
{\bf Claim}. {\em Let $\bar\si\in \gal(L/k(a))^e$, and suppose that $\Fix(\bar\si)$ is
regular over $k$. Then the number of $\bar\tau\in \gal(M/k(a))$
which restrict to $\bar\si$ and with $\Fix(\bar\tau)$ regular over $k$,
equals $\phi(k_M/k_L,k)[M:k_ML]^e$, and therefore does not depend on
$\bar\si$.}\\[0.05in] 
Indeed, let $M'_1,\ldots,M'_s$ be  a list (up to conjugation over $k(a)$) of all regular extensions of $k$ which are
contained in $M$, and contain $k(a)$, and $\gal(M/M'_j)$
$e$-generated.  Then given $K_i$,
if $\theta'_j$ denotes the formula expressing that the relative
algebraic closure of $k(a)$ inside $M$ is isomorphic to $M'_j$, then one
has, for $\bar\tau\in \gal(M/k(a))^e$ and $\bar\si$ its restriction to
$L$: $\Fix(\bar\tau)\cap L\simeq_{k(a)}\Fix(\bar\si)$ if and only if
$\Fix(\bar\tau)\simeq _{k(a)}M'_j$ and $M'_j\cap L\simeq
\Fix(\bar\si)$. I.e., computed in $M$, we will have 
$$\mu_V(\theta_i)=\frac{\sum_{j\in J_i}|\{\bar\tau\in \gal(M/k(a))^e\mid
  \Fix(\bar\tau)\simeq_{k(a)}M'_j\}|}{\sum_{1\leq j\leq r}|\{\bar\tau\in
  \gal(M/k(a)^e\mid \Fix(\bar\tau)\simeq_{k(a)}M'_j\}|},$$
where $J_i=\{j\mid M'_j\cap L\simeq K_i\}$. But by the claim, the number
on the right hand side equals 
$$\frac{|\{\bar\si\in
    \gal(L/k(a))^e\mid \Fix(\bar\si)\cap L\simeq_{k(a)}K_i\}|\phi(k_M/k_L,k)[M:k_ML]^e}{|\{\bar\si\in
    \gal(L/k(a))^e\mid \Fix(\bar\si)\cap
    k_L=k\}|\phi(k_M/k_L,k)[M:k_ML]^e}=\frac{\mu(\theta_i(a))}{\sum_j\mu(\theta_j(a))},$$
which equals $\mu_V(\theta_i)$ computed in $L$. It therefore suffices to prove the claim. \\[0.05in]
\noindent
{\em Proof of the claim}. Consider $\bar\si\in \gal(L/k(a))^e$, which projects onto a set of generators of
$\gal(k_L/k)$. Then the following two sets have
the same cardinality $\phi(k_M/k_L,k)$:
\begin{center}$\{\bar \tau\in \gal(k_M/k)^e\mid \bar\tau\rest_{k_L}=\bar\si\rest_{k_L}, \langle
\bar\tau\rangle =\gal(k_M/k)\}$\\
$\{\bar \tau\in \gal(k_ML/k(a))^e\mid \bar\tau\rest_L=\bar\si, \langle
\bar \tau\rest_{k_M}\rangle=\gal(k_M/k)\}$. \end{center}
Hence
$$|\{\bar\tau\in \gal(M/k(a))^e\mid \bar\tau\rest_L=\bar\si,
\langle\bar\tau\rest_{k_M}\rangle =\gal(k_M/k)\}|=\phi(k_M/k_L,k)[M:k_ML]^e.$$
\qed

\para {\bf Another way of counting}. Let $\bar\si_0\in \gal(k(a))^e$ be a
lift of an $e$-tuple generating $\gal(k)$. Then the set
$\{\bar\si_0\cdot\bar\tau\mid \bar\tau\in \gal(k^s(a))^e\}$ coincides with
the set $\{\bar\si\in \gal(k(a))^e\mid \fix(\bar\si)\cap k^s=k, \  
\bar\si\supset \bar\si_0\rest_{k^s}\}$. We fix a finite Galois extension $L$ of
$k(a)$ such that $\gal(L/k(a))$ splits, i.e.,
$$\gal(L/k(a))\simeq \gal(L/k_L(a))\rtimes \gal(k_L/k)$$
and without loss of generality, the restriction of $\bar\si_0$ to $L$
belongs to  $ \gal(k_L/k)\leq \gal(L/k(a))$. (The fact that $\gal(k(a))\simeq
\gal(k^s(a))\rtimes \gal(k)$ is well known: since
$\gal(k)$ is projective,  the restriction map $\gal(k(a))\to \gal(k)$
splits.) 

Working in $\gal(k^s(a))$, we denote the Jarden-Kiehne measure by $\bar \mu$. Then
the fields $K_i$ are contained in $\Fix(\bar\si_0)$. If $\theta'_i$ is
the $\call(k^s(a))$-sentence expressing that the relative algebraic
closure of $k^s(a)$ in $k^sL$ is isomorphic over $k^s(a)$ to $k^sK_i$, then
$\theta'_i$ is a sentence of $\call(k_L(a))$, and  we
 define
$$\bar\mu_V(\theta_i)=\bar\mu(\theta'_i)=|\{\bar\si\in
\gal(L/k_L(a))^e\mid \fix(\bar\si)\cap L\simeq k_LK_i\}|[L:k_L(a)]^{-e}.$$

\begin{prop} \vlabel{counting2} Assumptions and notation as above. Then $\bar\mu_V=\mu_V$. \end{prop}

\prf 
We have
$$|\{\bar \si\in \gal(L/k(a))^e\mid \Fix(\bar\si)\cap
k_L=k\}|=\phi(k_L)[L:k_L(a)]^e.$$
Note  that the
fields $K_i$ are linearly disjoint from $k^s(a)$ over $k(a)$. Hence, if
$K'$ is a regular extension of $k$ containing $k(a)$ and contained in
$L$, then $K'\simeq_{k(a)}K_i$ if and only if $k_{L}K'\simeq
_{k_L(a)}k_{L}K_i$. It follows that if $\bar\tau\in \gal(k^s(a))^e$, then \\
$$\fix(\bar\si_0\cdot \bar\tau)\cap L\simeq_{k(a)}K_i \iff
\fix(\bar\tau)\cap L\simeq_{k_L(a)}k_LK_i,$$
i.e.,
$$\Fix(\bar\si_0\cdot\bar\tau)\models \theta_i\iff \Fix(\bar\tau)\models \theta'_i.$$


\noindent
We then have $$|\{\bar\si\in \gal(L/k(a))^e\mid
\bar\si\rest_{k_L}=\bar\si_0\rest_{k_L},\ \Fix(\bar\si)\simeq_{k(a)}K_i\}|=[L:k_L(a)]^e\bar\mu(\theta'_i).$$
Hence, $\mu(\theta_i)=\phi(k_L)\bar\mu(\theta'_i)[L:k_L(a)]^e$, which gives $\mu_V(\theta_i)=\bar\mu_V(\theta_i)$.

\para Is there another way of counting, say via homomorphisms $\gal(k)\to
\gal(L/k(a))$, which induce the identity on $\gal(k_L/k)$? Recall that
$\gal(k^s(a))\simeq \gal(k^s(a))\rtimes \gal(k)$, and let $L$ be a finite Galois
extension of $k(a)$ such that $\gal(L/k(a))=\gal(L/k_L(a))\rtimes
\gal(k_L/k)$. Let $k^L$ be a Galois extension of $k$ such that any
homomorphism of $\gal(k)$ into $\gal(L/k(a))$, which induces the identity on
$\gal(k_L/k)$, factors through $\gal(k^L/k)$. Note that $k^L\supset k_L$. So, writing $\gal(k^LL/k(a))$ as
$\gal(k^LL/k^L(a))\rtimes \gal(k^L/k)$, we count the number of homomorphisms
$\gal(k^L/k)\to \gal(k^LL/k(a))$ which induce the identity on
$\gal(k^L/k)$. And among those, we count those whose fixed field is
isomorphic (over $k(a)$) to $K_i$. Here again $K_i$ runs over all
subextensions of $L/k(a)$ which are regular over $k$ and with
$\gal(L/K_i)$ a homomorphic image of $\gal(k)$. Given an $e$-tuple
$\bar\si_0$ generating $\gal(k^L/k)$, we are reduced to the previous
computation, since a homomorphism $\rho:\gal(k)\to \gal(k^LL/k(a))$ is uniquely
determined by $\rho(\bar\si_0)$, and $\rho(\bar\si_0)$ will be written
as $(\bar\si_0,\bar\tau)$ for some $\bar\tau\in \gal(L/k_L(a))^e$.


\subsection{Change of measure under definable maps}

\begin{lem}\vlabel{lem2} If $e>1$, then  the measure $\mu_V$ ($V\in {\rm Var}_k$) defined above is not
preserved under definable bijection. 
\end{lem}

\prf Let $a$ be a generic of $V$.  If $L$ is a finite Galois extension of $k(a)$, say with
$\gal(L/k(a))=\gal(L/k_L(a))\rtimes \gal(k_L/k)$, and
$\bar\si\in\gal(L/k(a))^e$ lifts a set of generators of $\gal(k_L/k)$,
we let $K:=\Fix(\bar\si)$. Then, in any perfect $e$-free PAC field $\calk$ containing
$k(a)$ and intersecting $L$ in $K$, we will have that $\dcl(ka)\cap L=
\Fix(N)$, where $N$ is the normaliser in in $\gal(L/k(a))$ of $H:=\langle
\bar\si\rangle$. 
 Indeed,
it is the subfield $E$ of $K$ fixed by the elements of $\aut(K/k(a))$, and if
$\rho\in \gal(L/k(a))$ restricts to an automorphism of $K$, this means
that it normalizes $H$, i.e., belongs to $N$. \\
Assume that $E\neq k(a)$, let $b\in E$ be such that $E=k(a,b)$, and
let $\pi:W\to V$ be the natural projection, where $W$ is the algebraic
locus of $(a,b)$ over $k$. Then $\pi$ defines a bijection between the generics $a'$ of $V$
which satisfy $\theta_K$ and the generics $(a',b')$ of $W$ which satisfy
$\theta^1_K$ (where $\theta^1_K$ is the formula expressing that the
relative algebraic closure of $k(a,b)$ in $L$ is isomorphic to $K$). By compactness, it defines a bijection between some
definable subset $S$ of $V(k)$ and some definable subset $S'$ of
$W(k)$. Let us now
count. \\
By definition (using the second way of counting), $\mu_V(S)=\bar\mu(\theta'_K)$, computed over $k^s(a)$, and
$\mu_W(S')=\bar\mu_1({\theta^1_K}')$, where $\bar\mu$ and $\bar\mu_1$ are the
Haar measures on $\gal(k^s(a))^e$ and 
$\gal(k^s(a,b))^e$ respectively, and ${\theta^1_K}'$ is the analogue of
$\theta'_K$ over $W$. We then have 
$$\mu_V(S)=\phi(k_LK,k_L(a))[N:H][L:k_L(a)]^{-e}$$
and $$\mu_W(S')=\phi(k_LK,k_L(a))[L:k_L(a,b)]^{-e}=\phi(k_LK,k_L(a))[k_L(a,b):k_L(a)]^e[L:k_L(a)]^{-e}.$$
Indeed, the number of field extensions within $L$ which are
$k_L(a)$-isomorphic to $K$ is $[N:H]$, and $K$ is Galois over
$k(a,b)$. From $[k_L(a,b):k_L(a)]=[N:H]$ we obtain the
result, i.e.
$$\mu_V(S)[N:H]^{e-1}=\mu_W(S').\leqno{(*)}$$
It therefore remains to find such an $(L,K,b)$. Let $L$ be a Galois
extension of $k(a)$, which is regular over $k$, and with Galois group
$S_n$, with $n\geq 5$. (Such extensions exist, see e.g. \cite[16.2.5(a)
and 16.2.6]{FJ}). The only non-trivial normal subgroup of $S_n$ is $A_n$, so take
any $\bar\si$ which generates a proper subgroup of $S_n$ not contained
in $A_n$ (e.g., all $\si_i=(1,2)$). Then the normalizer of
$\langle\bar\si\rangle$ is a proper non-trivial subgroup of
$\gal(L/k(a))=S_n$, and because $e>1$, $(*)$ gives the desired inequality. \qed

\begin{cor}\vlabel{cor1} Assume that $e=1$, and let $\mu_V$ be defined
as above for $V\in {\rm Var}_k$. Then
$\mu_V$ is preserved under definable bijections between sets of positive
measure; more precisely, let
$S\subset V$ and $S'\subset W$ be definable subsets of the varieties $V$
and $W$, and which are Zariski dense. If there is a definable bijection
$f$ 
between $S$ and $S'$ then $\mu_V(S)=\mu_W(S')$.
\end{cor}

\prf Let $a$ be a generic of $V$, and let $L$ be a finite Galois
extension of $k(a)$, such that whenever an $e$-free PAC field $\calk$
contains $a$ and is regular over $k$,
then whether $a$ belongs to $S$ or not is determined by  $L\cap \calk$. We may assume that
$\gal(L/k(a))=\gal(L/k_L(a))\rtimes \gal(k_L/k)$, and that for some subfield
$K$ of $L$, $a\in S(\calk)$ if and only if $\calk\cap
L\simeq_{k(a)}K$. Let $c=f(a)$, and let $b\in \calk\cap L$ generate
$\dcl(ka)\cap L$ over $k(a)$. Note that because $f$ is a bijection,
$\dcl(kc)=\dcl(ka)$. The result follows from $(*)$ in Lemma \ref{lem2}. \qed

\begin{example} (Another example). We will build an example of a
  definable bijection $f$ between two definable subsets of some $V \in {\rm Var}_{k}$ in a perfect $e$-free PAC field $k$ with $e > 1$
  which does not preserve the measure. 
\end{example}

\prf We assume $k$ contains a primitive $5$th root of $1$. Let $u,v\in
k^\times$, $a$ transcendental over $k$, and consider $u+av$, $u-av$, and
let $L=k(a,\sqrt[5]{u+av},\sqrt[5]{u-av})$. Then $L$ is Galois over
$k(a^2)$, with Galois group $\zee/5\zee \wr \zee/2\zee$. Let
$K=k(a,\sqrt[5]{u+av})$, and $\calk$ any $e$-free perfect PAC with $e > 1$, which is
regular over $k$ and intersects $L$ in $K$. Then, inside $\calk$, we
have $\dcl(ka^2)=\dcl(ka)=k(a)$. Indeed, $a$ satisfies the following
formula $$x^2=a^2\land \exists y\, y^5=u+xv,$$ 
while $-a$ satisfies its negation. \qed 

\begin{example} (And a third example). Assume now that $k$ does not
contain a primitive $5$th root $\zeta$ of $1$, and consider the map $f:x\mapsto
x^5$. It is an injective map on any regular extension of $k$. If
$L=k(a,\zeta,\sqrt[5]a)$, then $\gal(L/k(a))\simeq \zee/5\zee \rtimes
\gal(k(\zeta)/k)$. If $\si\in \gal(L/k(a))$ restricts to a generator of
$\gal(k(\zeta)/k)$, then its order equals $[k(\zeta):k]$. So, when $e=1$,
every $e$-free  field which is a regular extension of $k$ is closed
under $5$th roots. Suppose $e\geq 2$, and let $S$ be the image of $f$
in an $e$-free PAC $\calk$ which is regular over $k$. Then $\mu_{\gm}(S)=5^{1-e}$.
\end{example}


\begin{fact} \label{ImFact} (\cite{IH}, Theorem 17) 
Let $K$ be a perfect PAC field with pro-cyclic Galois group.  There is a unique function $\mu: \mathrm{Def}(K) \to \mathbb{Q}$ such that 
\begin{enumerate}
\item $\mu(K) = 1$.
\item $\mu$ is invariant under definable bijection.
\item If $X_{1}, X_{2} \in \mathrm{Def}(K)$ are disjoint, then $\mu(X_{1} \cup X_{2}) = \mu(X_{1})+\mu( X_{2})$ if $\mathrm{dim}X_{1} = \dim X_{2}$; otherwise, $\mu(X_{1} \cup X_{2}) = \max \{\mu (X_{1}), \mu(X_{2})\}$.  
\item Suppose $W \twoheadrightarrow^{G} V$ is a Galois cover with $W \in {\rm Var}_{K}$, and suppose $C \subseteq \mathrm{Hom}(\gal(K),G)$ is an orbit (under the action of $G$ by conjugation).  Then the measure of $X(W \twoheadrightarrow^{G} V,C)$ only depends on $C$ and $G$, as an abstract group. 
\end{enumerate}
\end{fact}

Here $X(W \twoheadrightarrow^{G} V,C)$ is a certain definable subset of the absolutely irreducible variety $V$. We refer the reader to \cite[Definitions 1 \& 17]{IH} for definitions of all the objects mentioned in Fact \ref{ImFact} and further details.  

\begin{prop} Let $e=1$. Then the measure $\mu_V$, $V\in {\rm Var}_k$, coincides with the
``non-standard counting measure'' defined in \cite{CDM}.
\end{prop}

\prf We will check that our measure satisfies the conditions of Fact \ref{ImFact}.  Conditions (1) and (3) are immediate. Condition (2) is Corollary \ref{cor1}. As for
  Condition (4), a straightforward inspection of the definitions gives
that our measure satisfies condition (4). Indeed, it follows from the following remark: let $f:W\to V$ be a Galois cover with
Galois group $G$ and with {$W$} absolutely irreducible. If $b$ is a
generic of $W$ over $K=k$, and $L=k(b)$, $a=f(b)$, then $a$ is a generic
of $V$, and $L\cap k^{alg}=k$,
so that if $K'$ is a subfield of $L$ with $\gal(L/K')$ pro-cyclic, then the
set $S$ of elements $a'\in V(k)$ which satisfy $\theta_{K'}$, has
measure which only depends on the pair $(G,\gal(L/K'))$, or equivalently,
on the pair $(G,C)$, where $C$ is the conjugacy class of a generator of
$\gal(L/K')$. That our measure satisfies (4) then follows from its
definition.  \qed

\begin{prop} Let $e>1$, let $k$ be a perfect $e$-free PAC field, let $V, W \in {\rm Var}_{k}$, and $f:S_1\to S_2$ a definable
bijection 
 between Zariski dense definable subsets $S_1$ of $V$ and $S_2$ of  $W$,  (so that
 $\dim(V)=\dim(W)$). Then there is a partition of $S_1$ into finitely
 many definable sets $U_1,\ldots,U_r$, and rational numbers
 $e_1,\ldots,e_r$, such that whenever 
  $D\subset U_i$ is definable,
 then $\mu_V(D)=e_i\mu_W(f(D))$. 
 \end{prop}

 \prf The Zariski closure of the graph of $f$ has finitely many
 irreducible components of dimension $\dim(V)$, and these in turn induce
 a partition of $S_1$ into finitely many definable subsets. We will
 therefore assume that the Zariski closure of the graph $\Gamma$ of $f$
 is irreducible. \\
 Let $(a,b)$ be a generic of $\Gamma$ over $k$, in some elementary
 extension of $k$, and with $f(a)=b$. By Corollary \ref{types}, there is a finite
 extension $L$ of $k(a,b)$ such that the isomorphism type over $k(a,b)$ of
 $k(a,b)^{alg}\cap L$ implies ``$f(a)=b$''. Let $c\in L$ be such
 that $\dcl(ka)\cap L=k(c)$. The computation made in Lemma \ref{lem2} then
 gives, for $R_1$ a definable subset of $S_1$:
 $$\mu_V(R_1)[k(c):k(a)]^{e-1}=\mu_W(f(R_1))[k(c):k(b)]^{e-1},$$
 i.e.,
 $$\mu_V(R_1)=\Bigl(\frac{[k(a,b):k(b)]}{[k(a,b):k(a)]}\Bigr)^{e-1}\mu_W(f(R_1)).$$

\begin{prop}\vlabel{rem1} If $V,W\in {\rm Var}_k$ are
birationally isomorphic by a map $f$, then $f$ preserves the measure:  if
$S\subset V$ is definable, then $\mu_V(S)=\mu_W(f(S))$. 
\end{prop}

\begin{proof}The rational map $f: V \to W$ map defines a
$k$-automorphism of $k(V)= k(W)$. As $k(V)$ is regular over $k$, $f$ 
extends to a $k^{alg}$-automorphism $\tilde f$ of $k(V)^{alg}$; then
$\tilde f$ induces a continuous automorphism of $\gal(k(V))$, which
induces the identity on $\gal(k)$. I.e., if $L$ is a finite Galois
extension of $k(V)$, then $\tilde f(L)$ is a finite Galois extension of
$k(W)$, which intersects $k^{alg}$ in $k_L=k^{alg}\cap L$, and
$\gal(\tilde f(L))/k(W))$ and $\gal(L/k(V))$ are isomorphic, by an
isomorphism which is the identity on $\gal(k_L/k)$.  This implies $\mu_{V}(S) = \mu_{W}(f(S))$.  
\end{proof}

{The following lemma records some closure properties of the family of definably amenable groups in an arbitrary theory.} 

\begin{lem} \label{closure props}
Suppose $G$ and $H$ are are definable groups.
\begin{enumerate}
\item If $H \leq G$ is a definably amenable subgroup of $G$ of finite index and $H$ is definably amenable, then $G$ is definably amenable.
\item If $\pi: G \to H$ is a definable surjective homomorphism with finite kernel and $H$ is definably amenable, then $G$ is definably amenable.
\end{enumerate}
\end{lem}

\prf 
(1) Let $n = [G: H]$ and let $g_{1}H, \ldots, g_{n}H$ list the left cosets of $H$ in $G$ with $g_{1} = e$.  Let $\mu_{H}$ denote the measure on $H$.  For a definable subset $X \subseteq G$, define $\mu_{G}$ by 
$$
\mu_{G}(X)  = \frac{1}{n}\sum_{i = 1}^{n} \mu_{H}(g_{i}^{-1}(X \cap g_{i}H)).
$$
It is clear that $\mu_{G}(G) = 1$, $\mu_{G}$ is $G$-invariant, and the finite additivity of $\mu_{G}$ follows from that of $\mu_{H}$.  

\noindent
(2) Let $\mu_{H}$ denote the measure witnessing the definable amenability of $H$ and set $n = |\Ker(\pi)|$.  Given a definable set $X \subseteq G$, define, for each $1 \leq i \leq n$, 
$$
(X)_{i} = \{x \in X : |X \cap \pi^{-1}(\pi(x))| = i\}.
$$
The $(X)_{i}$'s form a definable partition of $X$. We define $\mu_{G}$ by setting 
$$
\mu_{G}(X) =\frac{1}{n} \sum_{i = 1}^{n} i \mu_{H}(\pi((X)_{i})).
$$
This is finitely additive because $\mu_{H}$ is, and since $(G)_{n} = G$,
we have $\mu_{G}(G) = \frac{1}{n} (n \mu_{H}(H)) = 1$.  Invariance
follows from the left transitivity of $\mu_{H}$ and the fact that
$(gX)_{i} = g(X)_{i}$ for all $1 \leq i \leq n$. \qed

\begin{thm} \vlabel{thm2} Let $e\geq 1$, let $k$ be a perfect $e$-free PAC field, and
let $G$ be a group definable in $k$. Then $G$ is definably amenable. 
\end{thm}

\prf By Theorem C in \cite{HP1}, there is a definable subgroup $G_0$ of
finite index in $G$, and a definable homomorphism $\pi:G_0\to H(k)$, where
$H$ is a (connected) algebraic group, with $\Ker(\pi)$ finite, and
$\pi(G_0)$ Zariski dense in $H$. By Remark \ref{rem1}, the measure $\mu_H$ is stable under
translation, so $H(k)$ is definably amenable (via $\mu_H$). As $\pi(G)$ is definable and has finite
index in $H(k)$, so is $\pi(G)$. The conclusion then follows by Lemma \ref{closure props}.\qed  
\section{Extension of the measure to other fields}

\subsection{$\omega$-free perfect PAC fields}
\para {\bf Definition and useful facts}.  Recall that a
field is $\omega$-free if it has an elementary substructure with Galois
group isomorphic to $\hat F_\omega$, the free profinite group on
countably many generators. Equivalently, if every finite embedding
problem of its absolute Galois group has a solution, see \S3 in \cite{Jw}, or Chapters 25, 27 in
\cite{FJ}.\\[0.05in] 
The theory of perfect $\omega$-free PAC fields can be viewed as the
limit of the theories of perfect $e$-free PAC fields, see also Lemma
\ref{ultraproduct} below. Fact \ref{fact1}
and its Corollary \ref{types} both generalise to $\omega$-free PAC fields,
see \S4 in \cite{Jw}. 

\begin{lem}\vlabel{ultraproduct}
Suppose $K$ is an $\omega$-free perfect PAC field.  Then there is a collection
of PAC fields $(K_{e})_{e \in \mathbb{N}}$ and a non-principal
ultrafilter $\mathcal{U}$ on $\mathbb{N}$ such that each $K_{e}$ is a
PAC field with free absolute Galois group of  rank $e$ and $K \equiv
\prod_{e \in \mathbb{N}} K_{e} / \mathcal{U}$.  If $K$ is countable,
then $K$ embeds elementarily in $\prod_{e \in \mathbb{N}} K_{e} /
\mathcal{U}$. 
\end{lem}

\begin{proof}
We know that ${\rm Th}(K)$ is axiomatised within the class of perfect
$\omega$-free PAC fields by a description of the isomorphism type of its
absolute numbers (see Thm~4.2, \cite{Jw}). If the characteristic is positive, then the absolute Galois
group of the field of absolute numbers of $K$ is procyclic, and therefore there
is an $e$-free perfect PAC field $K_e$ with the same absolute numbers as
$K$ for every $e\geq 1$. \\
If the characteristic is $0$, however, it may be that the absolute
Galois group of the field of absolute numbers of $K$ is not finitely
generated. Let $\si_j,j\in\nat$, be a set of topological generators of
$\gal(K\cap \rat^{alg})$, and for each $e\geq 1$,
choose an $e$-free perfect PAC field $K_e$, satisfying $K_e\cap
\rat^{alg}=\fix(\si_0,\ldots,\si_e)$.

If $K^*$ is a non-principal ultraproduct of the $K_e$, then $K^*$ is
perfect 
PAC, with field of absolute numbers equal to $K\cap \rat^{alg}$, and
absolute Galois group free on infinitely many generators (\cite{FJ}, Thm
25.2.3), hence is $\omega$-free.

The last assertion follows from the fact that $\prod_{e \in \mathbb{N}}
K_{e} / \mathcal{U}$ is $\aleph_1$-saturated when $\calu$ is non-principal. 
\end{proof}

\para 
In \cite{Jw}, Jarden
introduces a measure on $\calg=\bigudot_{e\geq 1} \gal(K)^e$, for $K$ a
 field, as follows. First, one defines a topology on
$\calg$ by stating that a subset $A$ of $\calg$ is open if $A_e=A\cap
\gal(K)^e$ is open for all $e$. Then $\calg$ is Hausdorff, locally compact,
and totally disconnected. A closed subset $C$ is compact if and only if
it is bounded, i.e., $C_e=\emptyset$ for $e\gg 0$, and $\gal(K)$ acts
continuously on $\calg$ from the left and the right. A subset $A$ is
measurable if $A_e$ is measurable for all $e$ (for the Haar measure $\mu_e$ on
$\gal(K)^e$). If $A$ is measurable, then one defines
$$\mu(A)=\sum_{e=1}^\infty \mu_e(A_e).$$
Then $\mu$ is a complete regular Borel measure of $\calg$ and is invariant
under the action of $\gal(K)$. Note that many sets have infinite measure,
but some do not. One result we will use is the following:

\begin{fact}\vlabel{lemJ} (Lemma 7.2 in \cite{Jw}). Let $K$ be a countable Hilbertian
field, and $\theta$ an $\call(K)$-elementary statement of one variable
(i.e., a Boolean combinations of $\call(K)$-sentences of the form
$[\exists x\, f(x)=0]$, where $f\in K[X]$).
\begin{enumerate}
\item[(a)]
There are integers $n,r>0$
and {$1\leq n_1,\ldots,n_r\leq n$, $\varepsilon_1,\ldots,\varepsilon_r\in
\{\pm 1\}$} such
that $$\mu_e(\theta)=\sum_{i=1}^r {\varepsilon_i}\Bigl(\frac{n_i}{n}\Bigr)^e$$
for every $e>0$.
\item[(b)] If $K\not\models \theta$, then $n_i<n$ for
$i=1,\ldots,r$. Hence $\lim_{e\rightarrow
  +\infty} \mu_e(\theta)=0$, and $\mu(\theta)\in\rat$.
\item[(c)] If $K\models \theta$, then $\lim_{e\rightarrow
  +\infty}\mu_e(\theta)=1$, and $\mu(\theta)=\infty$. 
\end{enumerate}
\end{fact} 

\bigskip
Much of what follows can be found in Jarden's paper \cite{Jw} or in \cite{FJ}, at least
implicitly, but we
chose to give proofs.

\begin{lem}\vlabel{lemJreg}. (See also \S 20.7 in \cite{FJ})  Suppose $k_{0}$ is a perfect field and $V \in {\rm Var}_{k_{0}}$. Let $a$ be a generic of $V$ over
$k_0$.  
\begin{enumerate} 
\item Let $L$ be a finite Galois extension of $k_0(a)$. Let $K$ be a
subfield of $L$ containing $k_0(a)$,  which is regular over $k_0$, and
consider the formula $\theta_K$. 
 There are integers $n,r>0$
and $1\leq n_1,\ldots,n_r\leq n$, $\varepsilon_1,\ldots,\varepsilon_r\in
\{\pm 1\}$ such
that if $k$ is an $e$-free perfect PAC which is regular over 
$k_0$, then $$\mu_V(\theta_K)=\sum_{i=1}^r{\varepsilon_i}\Bigl(\frac{n_i}{n}\Bigr)^e.$$
Here $V$ is identified with $V_k$, and $\mu_V$ is the measure defined in
Paragraph 3.2. 
\item Let $S\subset V$ be definable with parameters in $k_0$ by a
formula $\varphi(x)$.  There are integers $n,r,e_1>0$
and $1\leq n_1,\ldots,n_r\leq n$, $\varepsilon_1,\ldots,\varepsilon_r\in
\{\pm 1\}$ such that whenever $k$ is an $e$-free perfect PAC field which
is regular over $k_0$
and $e\geq e_1$, then
$$\mu_V(\varphi)=\sum_{i=1}^r {\varepsilon_i}\Bigl(\frac{n_i}{n}\Bigr)^e.$$
If $\varphi$ is a Boolean combination of test formulas (cf \ref{theta}), then we may take
$e_1=1$.  

\end{enumerate}

\end{lem}

\prf Let us first show how (2)  follows from (1). The last assertion is  
clear, since such a formula can be written as a disjunction of mutually
incompatible test formulas. (Note however that the statement may be vacuous if $\gal(k_0)$ is not
$e$-generated.) The general case then follows: modulo the theory of
perfect $\omega$-free PAC fields which are regular extensions of $k_0$, we know
that the formula $\varphi$ is equivalent to a disjunction of mutually
incompatible test-formulas. By ultraproduct, the same holds modulo the
theory of perfect $e$-free PAC fields which are regular extensions of $k_0$ for
$e$ sufficiently large. We take $e_1$ be such that the equivalence holds
for all $e\geq e_1$.

\noindent 
(1) Let $K_1,\ldots,K_s$ enumerate all regular extensions of $k_0$ which
are between $k_0(a)$ and $L$. Let $k$ be an $e$-free perfect PAC field
regular over  $k_0$ and linearly disjoint from $k_0(a)$ over $k_0$. As we
saw in Proposition \ref{counting2}, if $K=K_i$,  $$\mu_V(\theta_K)=\bar
\mu_e(\theta_K)=\mu'_e(\theta_K),$$ where $\bar\mu_e$ is the
Jarden-Kiehne measure on $\gal(k^s(a))^e$ and $\mu'_e$ the Jarden-Kiehne
measure  on $\gal(k_0^s(a))^e$. The second equality is because $k$ and
$k_0(a)$ are linearly disjoint over $k_0$, so that $k^s$ and $k_0^sL$
are linearly disjoint (because free) over $k_0^s$, and therefore
$\gal(k^sL/k^s(a))\simeq \gal(k_0^sL/k_0^s(a))$. The result then follows
by Lemma \ref{lemJ}.\qed

\para {\bf A first definition of the measure}. 
Let $k$ be a perfect countable $\omega$-free PAC field, $V \in {\rm Var}_{k}$, and consider, for each $e\geq 1$,
the measure $\mu_{e,V}$ defined above. If $S$ is a definable subset of
$V$, defined by the formula $\varphi(x)$, we then set $\mu'_V(\varphi)=\sum_e\mu_{e,V}(\varphi)$.\\
As in Fact \ref{lemJ}, one of the consequences of 
  Lemma \ref{lemJreg}(2) is that if $k_0(a)\not\models\theta$, then
  $\mu'(\theta)\in\rat$, as we will see below. 
  The
proof goes as follows: if $\theta$ is determined by the finite Galois
extension $L$ of $k_0(a)$, where $a$ is a generic of $V$ and $k_0$ is
the relative algebraic closure in $k$ of the field of definition of $V$, then
$\mu_{e,V}(\theta)=\bar \mu_e(\theta)$, by
\ref{counting2}. Let $n=[k_0^{alg}L:k_0^{alg}(a)]$, $n_i,\varepsilon_i, 1\leq i\leq r$
and $e_1$
be given by Lemma \ref{lemJreg}. Then
$$\mu'_V(\theta)=\sum_{e=1}^{e_1-1}\bar\mu_e(\theta)+\sum_{i=1}^r\sum_{e\geq
                             e_1}\varepsilon_i\Bigl(\frac{n_i}{n}\Bigr)^e.$$
 
We may assume that the sum is reduced, i.e., that for no
$i\neq j$ we have $n_i=n_j$ and $\varepsilon_i\varepsilon_j=-1$. Note
that $\sum_{e=1}^{e_1-1}\bar\mu_e(\theta)\in\rat$, and that if $n_i<n$, then 
$\sum_{e\geq   e_1}\varepsilon_i\Bigl(\frac{n_i}{n}\Bigr)^e
 \in\rat$. However, any $i$ with $n_i=n$ will
  contribute to $+\infty$ (The measure being
  positive, and the sum reduced, if $n_i=n$,
  then $\epsilon_i=1$). 

If
$k_0(a)\models \theta(a)$, then all $\bar\mu_e(\theta)$ are equal to $1$. If
not, then all $n_i$ are $<n$, so that we get
$$\mu'_V(\theta)=(\sum_{e=1}^{e_1-1}\bar\mu_e(\theta))+\sum_{i=1}^r\varepsilon_i\Bigl(\frac{n_i}{n}\Bigr)^{e_1}\Bigl(\frac{n_i}{n-n_i}\Bigr).$$

\begin{cor} Let $k$ be a perfect $\omega$-free PAC field, $V \in {\rm Var}_{k}$, and $\mu'_V$ the measure defined above. Then
$\mu'_V(V)=\infty$. Let $S\subset V$ be definable by a 
formula $\varphi(x)$ over $k$.  If
$k(a)\not\models \varphi(a)$, 
then $\mu'_V(\varphi)\in\rat$. 
\end{cor}

\prf The result follows by Lemma \ref{lemJreg}.

\begin{example} Consider $V=\gm$, and $S$ the set of squares. Then
$\mu_{e,V}(S)=2^{-e}$, so that $\mu_V(S)=1/2$. We do have that
$k(a)\not\models \exists y\ y^2=a$; and therefore $\mu_V(V\setminus
S)=\infty$, one computes
that it equals $\sum 1^e - 2^{-e}=\infty$. 
\end{example}


\para{\bf Another definition of a measure}.  Let $k$ be a perfect
countable 
$\omega$-free PAC field, $V \in {\rm Var}_{k}$, let $K_e$ be a
family of perfect $e$-free PAC fields  and $\calu$ an ultrafilter such
that $k\prec \prod_eK_e/\calu$ (cf Lemma \ref{ultraproduct}), and let 
 $\mu_V$ be defined as the limit of the measures
$\mu_{V,e}$ along the ultrafilter $\calu$.

\begin{lem}
Then the limit measure $\mu_V$ only takes the values $0$
and $1$, and does not depend on $\calu$.

\end{lem}

\prf Indeed, we know that $\mu_{V,e}(\theta)=\sum_{i=1}^r
\varepsilon_i\Bigl(\frac{n_i}{[L:k_L(a)]}\Bigr)^e$ for some integers
$r,n_1,\ldots,n_r,\varepsilon_1,\ldots,\varepsilon_r$ and $e\gg 0$, and that $\mu_{V,e}(\theta)\in [0,1]$. Hence,
$$\mu_V(\theta)=\sum_{i=1}^r \lim_{e\rightarrow
  +\infty}\varepsilon_i\Bigl(\frac{n_i}{[L:k_L(a)]}\Bigr)^e.$$  Note that if
$n_i<[L:k_L(a)]$, then it will contribute $0$ to the sum, and if
$n_i=n$, it will contribute $1$. Hence $\mu_V(\theta)\in [0,1]\cap
\zee=\{0,1\}$. \qed 

\begin{rem}
This means that there is a unique {\em generic type} $p_V$ on the variety $V$:
If $a$ realises $p_V$ over $k$, then $a$ is a generic of $V$ over $k$,
and its relative algebraic closure in a model is $k(a)$.
\end{rem}

\begin{prop} Let $k$ be perfect $\omega$-free PAC, $H$ a
connected algebraic
  group defined over $k$, and $G\leq H(k)$ a definable subgroup. Then
  either $G=H(k)$, or $\mu_H(G)=0$, and $[H(k):G]=\infty$.
\end{prop}

\prf If $G\leq H(k)$ is proper, then $\mu_H(G)=[H(k):G]\inv<1$, so must
equal $0$.

\begin{prop} \label{omega-free definably amenable} Let $G$ be a group definable in a  $\omega$-free perfect PAC field
  $k$. Then $G$ is definably amenable.
\end{prop}

\prf By Lemma \ref{ultraproduct}, $G$ embeds elementarily in an ultraproduct
$\prod_{e\in\nat}G_e/\calu$, where $G_e$ is definable in the $e$-free perfect
PAC field $K_e$. Each $G_e $ is definably amenable, hence so is
$G$. \qed

\para{\bf Some questions}.\\[0.05in]
$\bullet$ What are the possible sets of values of $\mu_G$? \\
$\bullet$ Can this set be infinite? \\
$\bullet$ Can it contain irrational
numbers? \\
$\bullet$ Does it depend on the ultrafilter $\calu$?

\subsection{Perfect PAC fields with free pro-$p$ absolute Galois group}

\para{\bf Measure on pro-$p$-$e$-free PAC fields.} We consider the theory of
perfect PAC fields, with absolute Galois group free pro-$p$ on $e$
generators \textemdash i.e., the absolute Galois group is the pro-$p$-completion of
$F_e$. Let $k$ be a perfect field with absolute Galois group free
pro-$p$ on $e$ generators, let $V \in {\rm Var}_{k}$, and $a$ a generic of $V$ over $k$. Choose a $p$-Sylow
subgroup $P$  of $\gal(k(a))$, and let
$F=\Fix(P)$.\\
Let $L$ be a finite Galois
extension of $k(a)$, and $K$ a subfield of $L$ containing $k(a)$,
regular over $k$, and with $\gal(L/K)$ an $e$-generated $p$-group. We wish to define
$\mu(\theta_K)$. Since all $p$-Sylow of $\gal(k(a))$ are conjugate by an
element of $\gal(k(a))$, and $\gal(k^sL)$ is a normal subgroup of
$\gal(k(a))$, $P$ projects onto
$\gal(k^sL/K)$. We now set
$$\mu_V(\theta_K)=\frac{|\{\bar\si\in \gal(FL/F)^e\mid
  \Fix(\bar\si)\simeq_FFK\}| }{|\{\bar\si\in \gal(FL/F)^e\mid
  \Fix(\bar\si)\cap k^s=k\}|}.$$
Note that the choice of $P$ only depends on the field $k(a)$ (i.e., on
the variety $V$), and not on the fields $L,K$. We need to show that the
definition of $\mu_V$ does not depend on $L$, i.e., that if we compute it
in some $FM$, with $M$ Galois over $k(a)$ and containing $L$, we will
get the same number. But this follows from the following remarks: \\
Let $\bar\si\in \gal(FL/F)^e$ project onto a set of generators of
$\gal(k_L/k)$; then $\bar\si\rest_{k_L}$ has $\phi(k_M/k_L,k)$ many
extensions to $k_MF$ which project onto a set of generators of
$\gal(k_M/k)$, and therefore $\bar\si$ has $\phi(k_M/k_L,k)$ many
extensions to $k_MFL$ with the same property; since $\gal(FM/FL)$ is a $p$-group, $\bar\si$ has
exactly
$$\phi(k_M/k_L,k)[FM:k_MFL]^e$$
many extensions to $FM$ with fixed field a regular extension of $k$. The result follows.

\begin{cor}
Suppose $k$ is a perfect PAC field whose absolute Galois group is free
pro-$p$. Then  every group definable in $k$ is definably amenable.  
\end{cor}
\begin{proof}
In the case that the absolute Galois group of $k$ is free pro-$p$ of finite rank $e$, then, as $k$ is a bounded perfect PAC field, Theorem C of \cite{HP1} implies that $G$ is virtually isogenous to the $k$ points of an algebraic group $H$ and we may argue as in Theorem \ref{thm2}, using that the measure $\mu_{H}$ witnesses the definable amenability of $H(k)$.  If the absolute Galois group of $k$ is free pro-$p$ of infinite rank, then $k$ is elementary equivalent to a non-principal ultraproduct of fields $F_{e}$ with $F_{e}$ perfect PAC and pro-$p$-$e$.  We then obtain a translation invariant ultralimit measure as in Proposition \ref{omega-free definably amenable}.  
\end{proof}

\bigskip\noindent
DMA - Ecole Normale Sup\'erieure\\
45, rue d'Ulm\\
75230 Paris Cedex 05\\
France\\
e-mail: {\tt zoe.chatzidakis@ens.fr}\\[0.1in]

\noindent Department of Mathematics - UCLA\\
Math Sciences Building\\
520 Portola Plaza
Box 951555\\
Los Angeles, CA 90095\\
USA\\
e-mail: {\tt nickramsey@math.ucla.edu}

\end{document}